\documentclass{amsart}
\usepackage{amsmath}
\usepackage{amsthm}
\usepackage{amssymb}
\usepackage{mathrsfs}
\usepackage{hyperref}
\usepackage{chngcntr}
\usepackage{tikz}
\usetikzlibrary{matrix,arrows}

\counterwithin{equation}{section}

\makeatletter
\newtheorem*{rep@theorem}{\rep@title}
\newcommand{\newreptheorem}[2]{%
\newenvironment{rep#1}[1]{%
 \def\rep@title{#2 \ref{##1}}%
 \begin{rep@theorem}}%
 {\end{rep@theorem}}}
\makeatother

\newtheorem{thm}{Theorem}[section]
\newreptheorem{thm}{Theorem}
\newtheorem{prop}[thm]{Proposition} 
\newreptheorem{prop}{Proposition}

\newtheorem{lem}[thm]{Lemma}
\newreptheorem{lem}{Lemma}
\newtheorem{cor}[thm]{Corollary}

\theoremstyle{definition}

\newtheorem{exmpl}[thm]{Example}
\newtheorem{?}[thm]{Question}
\newtheorem*{blankdfn}{Definition}

\theoremstyle{remark}
\newtheorem{rmk}[thm]{Remark}

\newcommand{\HOM}{\mathbb{H}\text{om}}
\newcommand{\FR}{\mathfrak}
\newcommand{\ds}{\displaystyle}
\newcommand{\tql}{\textquotedblleft}
\newcommand{\tqr}{\textquotedblright}
\newcommand{\noin}{\noindent}

\begin{document}

\author{Scott Atkinson}
\title{Minimal Faces and Schur's Lemma for Embeddings into $R^\mathcal{U}$}
\address{Vanderbilt University, Nashville, TN, USA}
\email{scott.a.atkinson@vanderbilt.edu}

\begin{abstract}
In the context of N. Brown's $\HOM(N,R^\mathcal{U})$, we establish that given $\pi: N \rightarrow R^\mathcal{U}$, the dimension of the minimal face containing $[\pi]$ is one less than the dimension of the center of the relative commutant of $\pi$.  
We also show the \tql convex independence\tqr of extreme points in the sense that the convex hull of $n$ extreme points is an $n$-vertex simplex. Along the way, we establish a version of Schur's Lemma for embeddings of II$_1$-factors.
\end{abstract}
\maketitle

\section{Introduction and Main Results}

In \cite{topdyn}, N. Brown exhibits a convex structure on the space of unitary equivalence classes of embeddings of a II$_1$-factor in an ultrapower of the separable hyperfinite II$_1$-factor.  Since its appearance in 2011, the convex structure in \cite{topdyn} has received a fair amount of attention in the literature--see \cite{capfri}, \cite{brocap}, \cite{caprad}, \cite{chir}, \cite{paunescu1}, \cite{paunescu2}, \cite{caplup}, \cite{thuan}, and \cite{saa}. Part of the appeal of Brown's work is that it links convex geometric concepts with operator algebraic ones.  The purpose of the present paper is to deepen this connection.

Let $N$ be a separable II$_1$-factor, and let $R$ denote the separable hyperfinite II$_1$-factor.  We denote by $\HOM(N,R^\mathcal{U})$ the collection of unitary equivalence classes of $*$-homomorphisms $N \rightarrow R^\mathcal{U}$ where $\mathcal{U}$ is a free ultrafilter on the natural numbers. We let $[\pi]$ denote the equivalence class of the $*$-homomorphism $\pi: N \rightarrow R^\mathcal{U}$.  The work in \cite{topdyn} and \cite{capfri} demonstrates that $\HOM(N,R^\mathcal{U})$ can be considered as a closed bounded convex subset of a Banach space.

One of the main results in \cite{topdyn} is the following characterization of extreme points.

\begin{thm}[\cite{topdyn}]\label{brownchar}
The equivalence class $[\pi] \in \HOM(N,R^\mathcal{U})$ is extreme if and only if $\pi(N)' \cap R^\mathcal{U}$ is a factor.
\end{thm}

\noindent This is an important result because it gives a convex geometric perspective on the following well-known open question attributed to S. Popa: given any II$_1$-factor (that embeds into $R^\mathcal{U}$), does there exist an embedding $\pi: N \rightarrow R^\mathcal{U}$ such that its relative commutant is a factor?  In view of Theorem \ref{brownchar}, this question asks if the convex set $\HOM(N,R^\mathcal{U})$ is always guaranteed to have extreme points.

We offer a generalization of Theorem \ref{brownchar} by considering minimal faces in \linebreak $\HOM(N,R^\mathcal{U})$ and the algebraic information lying therein.  In particular, we will show that given an embedding $\pi: N \rightarrow R^\mathcal{U}$, the dimension of the minimal face in $\HOM(N,R^\mathcal{U})$ containing $[\pi]$ is directly related to the dimension of the center of the relative commutant of $\pi$. The latter usage of the word \tql dimension\tqr has the usual algebraic meaning; the former usage is made precise in the following definition.

\begin{blankdfn}
Given $[\pi] \in \HOM(N,R^\mathcal{U})$, let $F_{[\pi]}$ denote the minimal face in \linebreak $\HOM(N,R^\mathcal{U})$ containing $[\pi]$.  $F_{[\pi]}$ is obtained by intersecting all faces in $\HOM(N,R^\mathcal{U})$ that contain $[\pi]$.  Let $\dim(F_{[\pi]})$ be the dimension of the minimal face, given by the smallest $n$ such that $F_{[\pi]}$ affinely embeds into $\mathbb{R}^n$; if there is no such $n \in \mathbb{N}$, then we say $\dim(F_{[\pi]}) = \infty$.  As a convention, $\dim(F_{[\pi]}) = 0$ if and only if $F_{[\pi]} = \left\{[\pi]\right\}$ is a singleton--that is, $[\pi]$ is extreme.  We will use $\text{dim}(\cdot)$ to mean both the dimension of a minimal face, and the dimension of an algebra--the context will make the usage clear.
\end{blankdfn}


The main result of this paper is the following theorem.

\begin{thm}\label{simplexthm}
Let the embedding $\pi: N \rightarrow R^\mathcal{U}$ be given.
\begin{enumerate}

	\item\label{equation} $\dim(F_{[\pi]}) +1 = \dim(\mathcal{Z}(\pi(N)'\cap R^\mathcal{U}))$. Here $\mathcal{Z}(\pi(N)'\cap R^\mathcal{U})$ denotes the center of $\pi(N)'\cap R^\mathcal{U}$.

	\item\label{sim} If $\dim(\mathcal{Z}(\pi(N)'\cap R^\mathcal{U})) = n <\infty$ then $F_{[\pi]}$ is an $n$-vertex simplex.
	
	\item\label{iso} If $\varphi \in t_1[\pi_1] + \cdots + t_n[\pi_n]$ where $0< t_j < 1, \sum t_j = 1$, and $[\pi_j]$ is an extreme point for every $1\leq j \leq n$, then \[\varphi(N)' \cap R^\mathcal{U} \cong \oplus_{j=1}^n \pi_j(N)'\cap R^\mathcal{U}.\]	
	
\end{enumerate}
\end{thm}

\noindent To see that Theorem \ref{simplexthm} subsumes Theorem \ref{brownchar}, consider part \eqref{equation} in the case where \[\dim(F_{[\pi]}) +1 = \dim(\mathcal{Z}(\pi(N)'\cap R^\mathcal{U})) = 1.\]   

As mentioned above, the question of existence of extreme points in $\HOM(N,R^\mathcal{U})$ is a well-known open question that has recently received some attention in the literature. In \cite{caplup}, Capraro and Lupini observe that if $\HOM(N,R^\mathcal{U})$ embeds into either a dual Banach space or a strictly convex Banach space then the question has an affirmative answer. They go further to say that in \cite{chir}, Chirvasitu exhibits evidence supporting the possibility that $\HOM(N,R^\mathcal{U})$ does embed into a dual Banach space. To add to the list of reductions of this open question, Theorem \ref{simplexthm} immediately yields a convex geometric proof of the following corollary.

\begin{cor}\label{fidicenter}
If there is an embedding $\rho: N \rightarrow R^\mathcal{U}$ such that $\mathcal{Z}(\rho(N)' \cap R^\mathcal{U})$ is finite dimensional, then there is an embedding $\pi: N \rightarrow R^\mathcal{U}$ such that $\pi(N)'\cap R^\mathcal{U}$ is a factor. 
\end{cor}

\noin In fact, Corollary \ref{fidicenter} can be proved directly in the following weaker form: if there is an embedding $\rho: N \rightarrow R^\mathcal{U}$ such that $\mathcal{Z}(\rho(N)' \cap R^\mathcal{U})$ has a nonzero minimal central projection, then there is an embedding $\pi: N \rightarrow R^\mathcal{U}$ such that $\pi(N)'\cap R^\mathcal{U}$ is a factor. 


The following corollary indicates the \tql convex independence\tqr of extreme points.

\begin{cor}\label{linind}
The convex hull of $n$ extreme points in $\HOM(N,R^\mathcal{U})$ is always an $n$-vertex simplex and a face. 
\end{cor}

\noindent  For example, the convex hull of four extreme points cannot be a square--it must be a tetrahedron.

\begin{exmpl}
In Corollaries 6.10 and 6.11 of \cite{topdyn}, Brown exhibits II$_1$-factors with the property that for such a II$_1$-factor $N$, $\HOM(N,R^\mathcal{U})$ has infinitely many extreme points with a cluster point.  So for such a II$_1$-factor $N$ and any $n \in \mathbb{N}$, by Theorem \ref{simplexthm}, there is a face in $\HOM(N,R^\mathcal{U})$ taking the form of an $n$-vertex simplex. In fact, these can be nested.  
\end{exmpl}

\begin{rmk}
Simplices have the property that the convex hull of any finite number of extreme points is a simplex and a face.  Although $\HOM(N,R^\mathcal{U})$ is rarely a simplex (as it is rarely compact, see Theorem 4.7 of \cite{topdyn}), Theorem \ref{simplexthm} tells us that in the cases where extreme points exist, $\HOM(N,R^\mathcal{U})$ shares this property.
\end{rmk}

\begin{rmk}
At no point do we use that $N$ is a II$_1$-factor; so all of the results in this paper apply to $\HOM(\FR{A},R^\mathcal{U})$ for any separable unital $C^*$-algebra $\FR{A}$.  
Though it would require even more technical notation, it is reasonable to expect that these results extend further to $\HOM(\FR{A},M^\mathcal{U})$ for any separable unital $C^*$-algebra $\FR{A}$ and any separable McDuff II$_1$-factor $M$.  See \cite{saa} for more details.  These results may even extend to  P\u{a}unescu's convex structure on the space of sofic representations of a given sofic group appearing in \cite{paunescu1} and \cite{paunescu2}.  For the sake of simplicity, the arguments presented here will be limited to the context of \cite{topdyn}.  
\end{rmk}

A secondary result of this paper is the following theorem which can be thought of as a type of Schur's lemma for this context of embeddings of II$_1$-factors.

\begin{thm}\label{hardlemma}
Let $[\pi],[\rho] \in \HOM(N,R^\mathcal{U})$ be extreme points.  If there is a nonzero intertwiner $x \in R^\mathcal{U}$ such that $\pi(a)x=x\rho(a)$ for every $a \in N$, then $[\pi]=[\rho]$.
\end{thm}

\noin In the convex geometric context of $\HOM(N,R^\mathcal{U})$, extreme points are the irreducible/simple objects. This theorem is saying that there are no nonzero intertwiners between two inequivalent (representatives of) irreducibles.  From this perspective, it can be seen why one would consider Theorem \ref{hardlemma} as a sort of Schur's Lemma in the context of this article. It should be mentioned that Theorem \ref{hardlemma} can be stated in more general terms (see Remark \ref{generalschur}).


\subsection*{Acknowledgments} The author would like to thank Nate Brown, Jesse Peterson, David Sherman, and Stuart White for providing valuable comments and suggestions.  Gratitude is also due to the University of Virginia and the Hausdorff Research Institute for Mathematics for financial support during the development and writing of this article.

\section{Survey of $\HOM(N,R^\mathcal{U})$}\label{homnru}



We start with a technical proposition that is fundamental in the structure and analysis of $\HOM(N,R^\mathcal{U})$.  The following proposition appears as Proposition 3.1.2 of \cite{topdyn}.

\begin{prop}[\cite{topdyn}]\label{3.1.2}
Let $p,q \in R^\mathcal{U}$ be projections of the same trace, $M \subset pR^\mathcal{U}p$ be a \emph{separable} von Neumann subalgebra and $\chi: pR^\mathcal{U}p \rightarrow qR^\mathcal{U}q$ be a unital $*$-homomorphism.  Assume there exist projections $p_i,q_i \in R, i \in \mathbb{N}$ with $\tau(p_i) = \tau(q_i) = \tau(p)$ for every $i \in \mathbb{N}$ such that $(p_i)_\mathcal{U} = p$ and $(q_i)_\mathcal{U} = q$, and there exist $*$-homomorphisms $\chi_i: p_iRp_i \rightarrow q_iRq_i$ such that $\chi = (\chi_i)_\mathcal{U}$.  Then there exists a partial isometry $v \in R^\mathcal{U}$ with $v^*v = p$ and $vv^*=q$ such that $\chi(x) = vxv^*$ for every $x \in M$.
\end{prop}

\noin We will say that such a homomorphism $\chi$ lifts to coordinatewise homomorphisms or is liftable. Let $\sigma: R\otimes R \rightarrow R$ be an isomorphism, and to allow an abuse of notation, let $\sigma: (R\otimes R)^\mathcal{U} \rightarrow R^\mathcal{U}$ also denote the induced isomorphism between ultrapowers.  Using Proposition \ref{3.1.2}, it can be shown that given any $\pi: N \rightarrow R^\mathcal{U}, \pi$ is unitarily equivalent to $\sigma(1\otimes \pi)$ where $\sigma(1\otimes \pi)(x) = \sigma(1\otimes \pi(x))$.

Before exhibiting a convex structure on $\HOM(N,R^\mathcal{U})$, Brown establishes in \cite{topdyn} what it means to have a convex structure with out an ambient linear space.  Brown gives five axioms in Definition 2.1 of \cite{topdyn} that should be expected of a bounded convex subset of a linear space.  In \cite{capfri}, Capraro and Fritz show that closed bounded convex subsets of Banach spaces are characterized by these axioms defining a convex-like structure.  

In order to define convex combinations, we use the notion of a standard isomorphism. Let $p \in R^\mathcal{U}$ be a projection such that $p = (p_i)_\mathcal{U}$ where $p_i$ is a projection in $R$ with $\tau(p_i) = \tau(p)$ for each $i \in \mathbb{N}$.  An isomorphism $\theta_p: pR^\mathcal{U}p \rightarrow R^\mathcal{U}$ is called a \emph{standard isomorphism} if it lifts to coordinatewise isomorphisms $p_iRp_i \rightarrow R$.  We can define convex combinations as follows. Given $[\pi_1], \dots, [\pi_n] \in \HOM(N,R^\mathcal{U})$ and $0 \leq t_1,\dots, t_n \leq 1$ with $\sum t_i = 1$, we define 
\[t_1[\pi_1] + \cdots + t_n[\pi_n] := [\theta_{p_1}^{-1} \circ \pi_1 + \cdots + \theta_{p_n}^{-1}\circ \pi_n]\]
where for every $1 \leq i \leq n$ $\tau(p_i) = t_i$ and $\theta_{p_i}$ is a standard isomorphism. Thanks to Proposition \ref{3.1.2}, this operation is well-defined.  Proposition \ref{3.1.2} can also be used to show that 
\[t_1[\pi_1] + \cdots t_n[\pi_n] = [\sigma(p_1\otimes \pi_1) + \cdots + \sigma(p_n \otimes \pi_n)]\] where $p_1,\dots, p_n$ are projections with traces $t_1,\dots,t_n$ respectively and $\sigma(p_k \otimes \pi_k) (x) = \sigma  (p_k \otimes \pi_k(x))$ (see Example 4.5 of \cite{topdyn}). Under this definition, $\HOM(N,R^\mathcal{U})$ satisfies the axioms for a convex-like structure.  

Given $\pi: N\rightarrow R^\mathcal{U}$ and a projection $p \in \pi(N)'\cap R^\mathcal{U}$, we define the \emph{cutdown of $\pi$ by $p$} to be the map $\pi_p$ given by $\pi_p(x) = \theta_p(p\pi(x))$ where $\theta_p$ is a standard isomorphism.  It can be shown that $[\pi_p]$ is independent of the choice of the standard isomorphism. We record some important facts about cutdowns in the following proposition.

\begin{prop}[\cite{topdyn}]\label{cutdownprop}
Let $\pi: N\rightarrow R^\mathcal{U}$ be given.

\begin{enumerate}
\item Let $p \in \pi(N)'\cap R^\mathcal{U}$ be a projection. If $u \in R^\mathcal{U}$ is a unitary, then \[[\pi_p] = [(\text{Ad}(u) \circ \pi)_{upu^*}].\]

\item For any projection $p \in R^\mathcal{U}$, \[[\pi] = [\sigma(1\otimes \pi)_{\sigma(p\otimes 1)}].\]

\item \begin{enumerate}

	\item Given any $p \in \pi(N)'\cap R^\mathcal{U}$, $[\pi] = \tau(p)[\pi_p] + \tau(p^\perp)[\pi_{p^\perp}]$.
	
	\item If $[\pi] = t[\rho_1] + (1-t)[\rho_2]$ then there is a projection $p \in \pi(N)'\cap R^\mathcal{U}$ with trace $t$ such that $[\rho_1] = [\pi_p]$ and $[\rho_2] = [\pi_{p^\perp}]$.

\end{enumerate}

\item Let $p,q \in \pi(N)'\cap R^\mathcal{U}$ be projections with the same trace.  Then the following are equivalent.
	\begin{enumerate}
		\item $[\pi_p] = [\pi_q]$
		\item $p$ and $q$ are Murray-von Neumann equivalent in $\pi(N)'\cap R^\mathcal{U}$.
	\end{enumerate}
\end{enumerate}
\end{prop}

\noindent Theorem \ref{brownchar} follows quickly from Proposition \ref{cutdownprop}.

\section{Proof of Part \eqref{sim} of Theorem \ref{simplexthm}}

Our strategy to prove Theorem \ref{simplexthm} is to first prove part \eqref{sim} and then part \eqref{iso}.  Part \eqref{equation} will follow quickly from parts \eqref{sim} and \eqref{iso}.  From now on, fix a separable II$_1$-factor $N$ and $\pi: N \rightarrow R^\mathcal{U}$. The following proposition is a scaled version of part (2) of Proposition \ref{cutdownprop}.

\begin{prop}\label{rescale}
Let $p$ be a projection in $\pi(N)'\cap R^\mathcal{U}$. Then for any nonzero projection $Q \in R^\mathcal{U}$, we have \[[\pi_p] =[\sigma(1\otimes \pi)_{\sigma(1\otimes p)}] = [\sigma(1\otimes \pi)_{\sigma(Q \otimes p)}].\]
\end{prop}

\begin{proof}
To show the first equality, by Proposition \ref{3.1.2} there is a unitary $u \in R^\mathcal{U}$ so that $\sigma(1 \otimes \pi)(x) = u x u^*$ for every $x \in W^*(\pi(N) \cup \left\{p\right\})$.  Then by Proposition \ref{cutdownprop}, we have \[[\pi_p] = [(\text{Ad}u \circ \pi)_{upu^*}] = [\sigma(1 \otimes \pi)_{\sigma(1 \otimes p)}].\]

For the second equality, take $\theta_{\sigma(Q' \otimes p)} = \sigma\circ(\theta_{Q'} \otimes \theta_p)\circ \sigma^{-1}$ for any projection $Q'$ (see Definition 3.3.2 in \cite{topdyn}).\qedhere

\end{proof}

\noindent Note that one can only scale down a priori.  The next proposition shows that the collection of cutdowns of $\pi$ is convex.

\begin{prop}\label{cutdownsconvex} 
Let $p,q \in \pi(N)'\cap R^\mathcal{U}$ be projections with $\tau(p) = \tau(q)$.
\begin{align*}
t[\pi_p] + (1-t)[\pi_q] &= [\sigma(1\otimes \pi)_{(\sigma(S \otimes p) + \sigma(S^\perp \otimes q))}]
\end{align*}
for any projection $S \in R^\mathcal{U}$ with $\tau(S) = t$.
\end{prop}

\begin{proof}
We have that \[t[\pi_p] + (1-t)[\pi_q] = [\sigma(S \otimes \pi_p) + \sigma(S^\perp \otimes \pi_q)]\] for any projection $S \in R^\mathcal{U}$ with $\tau(S) = t$.  So we must show \[[\sigma(1\otimes \pi)_{\sigma(S \otimes p) + \sigma(S^\perp \otimes q)}] = [\sigma(S \otimes \pi_p) + \sigma(S^\perp \otimes \pi_q)].\]

By definition, for any $x \in N$, \[\sigma(1\otimes \pi)_{\sigma(S \otimes p) + \sigma(S^\perp \otimes q)}(x) = \theta_{\sigma(S \otimes p) + \sigma(S^\perp \otimes q)}(\sigma(S \otimes p\pi(x)) + \sigma(S^\perp \otimes q\pi(x)))\] and 
\[\sigma(S\otimes \pi_p(x)) + \sigma(S^\perp \otimes \pi_q(x)) = \sigma(S\otimes \theta_p(p\pi(x))) + \sigma(S^\perp \otimes \theta_q(q\pi(x))).\]
Now note that $\tau(\theta_{\sigma(S \otimes p) + \sigma(S^\perp \otimes q)}(\sigma(S\otimes p)) = \tau(S)$.
Put \[p':= \theta_{\sigma(S \otimes p) + \sigma(S^\perp \otimes q)}(\sigma(S\otimes p))\] and consider $\psi: p'R^\mathcal{U}p' \rightarrow \sigma(S \otimes 1) R^\mathcal{U} \sigma(S \otimes 1)$ given by \[\psi = \sigma \circ (S\otimes \text{id}) \circ \theta_{\sigma(S\otimes p)} \circ \theta_{\sigma(S\otimes p) + \sigma(S^\perp \otimes q)}^{-1}\big|_{p'R^\mathcal{U}p'}.\]

Let $u \in R^\mathcal{U}$ be a unitary such that for every $x \in N,$ $\sigma(1 \otimes \theta_p(p\pi(x))) = u \theta_p(p\pi(x))u^*.$ So, for $x \in N$ we have
\[\psi(\theta_{\sigma(S\otimes p) + \sigma(S^\perp \otimes q)}(\sigma(S\otimes p\pi(x)))) = \sigma(S\otimes u)\sigma(S \otimes \theta_p(p\pi(x)))\sigma(S\otimes u^*). \]


Evidently, $\psi$ is a unital $*$-homomorphism that lifts to coordinate-wise homorphisms. Then by Proposition \ref{3.1.2} there is a partial isometry $v \in R^\mathcal{U}$ such that $v^*v = p', vv^* = \sigma(S\otimes 1)$, and $\psi(x) = vxv^*$ for every \[x \in \theta_{\sigma(S\otimes p) + \sigma(S^\perp \otimes q)}(\sigma(S\otimes p\pi(N))).\]  Therefore, for every $x \in N$,
\[v^*\sigma(S\otimes u)\sigma(S \otimes \theta_p(p\pi(x)))\sigma(S\otimes u^*)v = \theta_{\sigma(S\otimes p) + \sigma(S^\perp \otimes q)}(\sigma(S\otimes p\pi(x))).\]
Let $v' := v^*\sigma(S\otimes u)$.  Then $v'^*v' = \sigma(S\otimes 1)$ and $v'v'^* = \theta_{\sigma(S\otimes p) + \sigma(S^\perp \otimes q)}(S\otimes p)$.  Thus \[v'\sigma(S\otimes \theta_p(p\pi(x))) v'^* = \theta_{\sigma(S\otimes p) + \sigma(S^\perp \otimes q)}(\sigma(S\otimes p\pi(x)))\] for every $x \in N$.

Similarly, there is a partial isometry $w' \in R^\mathcal{U}$ with $w'^*w' = \sigma(S^\perp \otimes 1)$ and $w'w'^* = \theta_{\sigma(S\otimes p) + \sigma(S^\perp \otimes q)}(\sigma(S^\perp \otimes q))$ such that \[w'\sigma(S^\perp \otimes \theta_q(q\pi(x))) w'^* = \theta_{\sigma(S\otimes p) + \sigma(S^\perp \otimes q)}(\sigma(S^\perp \otimes q\pi(x)))\] for every $x \in N$.

Thus, if $u' = v' + w'$ then $u'$ is a unitary such that \[u'(\sigma(S\otimes \theta_p(p\pi(a))) + \sigma(S^\perp \otimes \theta_q(q\pi(a))))u'^*\]\[ = \theta_{\sigma(S\otimes p) + \sigma(S^\perp \otimes q)}(\sigma(S\otimes p\pi(a)) + \sigma(S^\perp \otimes q\pi(a))).\qedhere\]
\end{proof}

\noindent Note that thanks to Proposition \ref{rescale}, the requirement that $p$ and $q$ have matching traces in Proposition \ref{cutdownsconvex} is no obstruction.


\begin{prop}\label{minface1}
\[F_{[\pi]} = \left\{[\pi_p] : p \in \pi(N)' \cap R^\mathcal{U}, \text{ a nonzero projection}\right\}\]
\end{prop}

\begin{proof}
%
This follows quickly from Proposition \ref{cutdownprop}, Proposition \ref{cutdownsconvex}, and the minimality of $F_{[\pi]}$.
\end{proof}

\begin{prop}\label{tensorcenter}
Let $\mathcal{Z}(\pi(N)'\cap R^\mathcal{U})$ be separable. If $z$ is a minimal central projection in $\mathcal{Z}(\pi(N)'\cap R^\mathcal{U})$ then $\sigma(1\otimes z)$ is minimal in $\mathcal{Z}(\sigma(1\otimes \pi)(N)'\cap R^\mathcal{U})$.	
\end{prop}

\begin{proof}
%
This follows immediately using Proposition \ref{3.1.2}.
\end{proof}


\begin{lem}\label{minface2}
Let $\mathcal{Z}(\pi(N)'\cap R^\mathcal{U})$ be finite dimensional with minimal central projections $z_1,\dots,z_n$, and let $0<t_0 \leq \min\left\{\tau(z_1),\dots,\tau(z_n)\right\}$.  Then \[F_{[\pi]} = \left\{[\pi_p] : p \in \pi(N)' \cap R^\mathcal{U}, \text{ a projection }, \tau(p) = t_0 \right\}.\]
\end{lem}

\begin{proof}
 By Proposition \ref{minface1} it suffices to show that for any projection $q \in \pi(N)'\cap R^\mathcal{U}$, there is a projection $p \in \pi(N)'\cap R^\mathcal{U}$ with trace $t_0$ such that $[\pi_q] = [\pi_p]$. Let $q \in \pi(N)' \cap R^\mathcal{U}$ be a projection and let $t' = \tau(q)$. Put
\[A_{t_0} = \left\{[\pi_p] : p \in \pi(N)' \cap R^\mathcal{U}, \text{ a projection }, \tau(p) = t_0 \right\}.\]
Assume that $t' > t_0$.  Let $Q \in R^\mathcal{U}$ be a projection with $\ds \tau(Q) = \frac{t_0}{t'}$, and let $u \in R^\mathcal{U}$ be a unitary such that $\sigma(1\otimes x) = uxu^*$ for every $x \in W^*(\pi(N) \cup \mathcal{Z}(\pi(N)'\cap R^\mathcal{U}))$, then by Proposition \ref{rescale},  \[[\pi_q] = [\sigma(1 \otimes \pi)_{\sigma(Q \otimes q)}] = [\pi_{u^*\sigma(Q \otimes q)u}] \in A_{t_0}.\]

Now let $t' < t_0$.  Let $p \in \pi(N)'\cap R^\mathcal{U}$ be a projection such that $\ds \tau(pz_i) = \frac{t_0}{t'}\tau(qz_i)$ for every $1 \leq i \leq n$.  Let $Q \in R^\mathcal{U}$ be a projection such that $\ds \tau(Q) = \frac{t'}{t_0}$.  By Proposition \ref{tensorcenter}, the minimal central projections in $\sigma(1 \otimes \pi)(N)'\cap R^\mathcal{U}$ are $\left\{\sigma(1 \otimes z_i)\right\}_{i=1}^n$.    Observe that for every $1 \leq i \leq n$ we have 
\[\tau(\sigma(Q \otimes p)\sigma(1\otimes z_i)) = \tau(\sigma(1\otimes q)\sigma(1\otimes z_i)).\]
Thus $\sigma(Q\otimes p)$ is Murray-von Neumann equivalent to $\sigma(1 \otimes q)$ in $\sigma(1\otimes \pi)(N)' \cap R^\mathcal{U}$. By Propositions \ref{cutdownprop} and \ref{rescale} we get that 
\[ [\pi_q] = [\sigma(1\otimes \pi)_{\sigma(1\otimes q)}]=[\sigma(1\otimes \pi)_{\sigma(Q \otimes p)}]= [\sigma(1\otimes \pi)_{\sigma(1\otimes p)}]= [\pi_p] \in A_{t_0}.\qedhere\]
\end{proof}


\begin{proof}(of part \eqref{sim} of Theorem \ref{simplexthm}) 
We will show that if $\mathcal{Z}(\pi(N)' \cap R^\mathcal{U})$ is $n$-dimensional with $n < \infty$ then $F_{[\pi]}$ is affinely isomorphic to the $n$-vertex simplex given by \[ \Delta_{n-1}^{t_0} := \left\{(x_1,\dots,x_n): 0\leq x_i \leq t_0  \quad \forall 1 \leq i \leq n, \sum_{i=1}^n x_i = t_0\right\}.\]  By Lemma \ref{minface2}, we may identify $F_{[\pi]}$ with \[A_{t_0} := \left\{[\pi_p] : p \in \pi(N)' \cap R^\mathcal{U}, \text{ a projection }, \tau(p) = t_0 \right\}.\]
Consider the map $\psi: A_{t_0} \rightarrow \Delta_{n-1}^{t_0}$ given by \[\psi([\pi_p]) = (\tau(pz_1), \dots, \tau(pz_n))\] where $z_1,\dots z_n$ are the minimal central projections of $\pi(N)' \cap R^\mathcal{U}$.  Proposition \ref{cutdownprop} ensures that $\psi$ is well-defined and injective.  Given any $(x_1,\dots, x_n) \in \Delta_{n-1}^{t_0}$, it is well-known that there is a projection $p \in \pi(N)'\cap R^\mathcal{U}$ such that $(\tau(pz_1),\linebreak \dots,\tau(pz_n)) = (x_1,\dots, x_n);$ thus, $\psi$ is surjective.  It remains to show that $\psi$ is affine.  Let $[\pi_p], [\pi_q] \in A_{t_0}$. Since $W^*(\pi(N) \cup \mathcal{Z}(\pi(N)'\cap R^\mathcal{U}))$ is separable, there is a unitary $u \in R^\mathcal{U}$ so that $\sigma(1 \otimes x) = uxu^*$ for every $x \in W^*(\pi(N) \cup \mathcal{Z}(\pi(N)'\cap R^\mathcal{U}))$.  Now, by Proposition \ref{cutdownsconvex}, \[t[\pi_p] + (1-t)[\pi_q] = [\sigma(1\otimes \pi)_{\sigma(S \otimes p) + \sigma(S^\perp \otimes q)}] = [\pi_{u^*(\sigma(S \otimes p) + \sigma(S^\perp \otimes q))u}]\] where $S \in R^\mathcal{U}$ is a projection such that $\tau(S)=t$.  Furthermore, for every $1 \leq i \leq n$, we have that 
\[\tau(u^*(\sigma(S \otimes p) + \sigma(S^\perp \otimes q))uz_i) =  t\tau(pz_i) + (1-t)\tau(qz_i).\]
So
\begin{align*}
\psi(t[\pi_p] + (1-t)[\pi_q]) &= \psi([\pi_{u^*(\sigma(S \otimes p) + \sigma(S^\perp \otimes q))u}]\\
&= t\psi([\pi_p]) + (1-t)\psi([\pi_q]).\qedhere
\end{align*}
\end{proof}

\section{Schur's Lemma for $\HOM(N,R^\mathcal{U})$}

By presenting a dichotomy for intertwiners of irreducibles, Theorem \ref{hardlemma} can be considered as a sort of $R^\mathcal{U}$-version of Schur's lemma.  


\begin{repthm}{hardlemma}
Let $[\pi],[\rho] \in \HOM(N,R^\mathcal{U})$ be extreme points.  If there is a nonzero intertwiner $x \in R^\mathcal{U}$ such that $\pi(a)x=x\rho(a)$ for every $a \in N$, then $[\pi]=[\rho]$.
\end{repthm}

\begin{proof}
Let $x = v|x|$ be the polar decomposition of $x$ (here $|x| = (x^*x)^\frac{1}{2}$).  It is a direct exercise to show that $v$ also intertwines $\pi$ and $\rho$.  Let $p = vv^* \in \pi(N)'\cap R^\mathcal{U}$ and $q = v^*v \in \rho(N)' \cap R^\mathcal{U}$.  Consider the liftable unital $*$-homomorphism \[\chi:= \theta_p \circ \text{Ad}(v) \circ \theta_q^{-1}: R^\mathcal{U} \rightarrow R^\mathcal{U}.\]  By Proposition \ref{3.1.2} there is a unitary $u \in R^\mathcal{U}$ such that $\chi(x) = uxu^*$ for every $x \in \rho_q(N)$.  Then we have for every $a \in N, \pi_p(a) = u\rho_q(a)u^*$.  Thus \[[\pi] = [\pi_p] = [\rho_q] = [\rho].\qedhere\]

\end{proof}

\begin{rmk}\label{generalschur}
Thanks to a proof suggested by S. White, Theorem \ref{hardlemma} can be stated in the following more general terms: Let $N$ be a unital $C^*$-algebra, let $M$ be a II$_1$-factor, and let $\pi,\rho: N \rightarrow M$ be unital $*$-homomorphisms so that $\pi(N)'\cap M$ and $\rho(N)'\cap M$ are both diffuse factors.  If there is a nonzero intertwiner $x \in M$ such that $\pi(a)x = x \rho(a)$ for every $a \in N$, then $\pi$ and $\rho$ are unitarily equivalent.  
\end{rmk}

\noindent Next we record the following easy corollary.  This is essentially a scaled version of Theorem \ref{hardlemma}.

\begin{cor}\label{scaletwine}
Let $p,q \in R^\mathcal{U}$ be mutually orthogonal projections with $\tau(p)=\tau(q)$.  Let $[\pi],[\rho]\in \HOM(N,R^\mathcal{U})$ be distinct extreme points.  If $x \in (p+q)R^\mathcal{U}(p+q)$ intertwines $\theta_p^{-1}\circ \pi$ and $\theta_q^{-1}\circ \rho$, then $x=0$.
\end{cor}


\section{Proofs of Parts \eqref{iso} and \eqref{equation} of Theorem \ref{simplexthm}}


\begin{proof}(of part \eqref{iso} of Theorem \ref{simplexthm})

We prove this part of the theorem in the case where $n=2$.  All other cases are direct generalizations of this one.  We must show that if $[\pi],[\rho] \in \HOM(N,R^\mathcal{U})$ are distinct extreme points and if $\varphi \in t[\pi] + (1-t)[\rho]$ for $0<t<1$, then \[\varphi(N)'\cap R^\mathcal{U} \cong \pi(N)'\cap R^\mathcal{U} \oplus \rho(N)'\cap R^\mathcal{U}.\]

Fix $K \in \mathbb{N}$ such that $\ds \frac{1}{K} < t$, and let $1 \leq k \leq K-1$ be such that $\ds \frac{k}{K} \leq t < \frac{k+1}{K}$.  Let $p \in R^\mathcal{U}$ be a projection with $\tau(p) = t$ and let \[\varphi = \theta_p^{-1} \circ \pi + \theta_{p^\perp}^{-1} \circ \rho.\]  Let $p_1,\dots,p_k, \tilde{p} \leq p$ be mutually orthogonal projections such that $\ds \tau(p_i) = \frac{1}{K}$ for $1 \leq i \leq k$ and $\ds\tau(\tilde{p}) = t - \frac{k}{K} (< \frac{1}{K})$; and let $v$ be a partial isometry (provided by Proposition \ref{3.1.2}) with $v^*v = vv^* = p$ such that \[\theta_p^{-1}\circ \pi = \text{Ad}(v) \circ \Big(\sum_{i=1}^k \theta_{p_i}^{-1} \circ \pi + \theta_{\tilde{p}}^{-1} \circ \pi\Big).\] Similarly, let $q_1,\dots,q_{K-k-1},\tilde{q} \leq p^\perp$ be mutually orthogonal projections such that $\ds \tau(q_j) = \frac{1}{K}$ for every $1 \leq j \leq K-k-1$ and $\ds \tau(\tilde{q}) = \frac{k+1}{K} - t$; and let $w$ be a partial isometry with $w^*w = ww^* = p^\perp$ such that \[\theta_{p^\perp}^{-1} \circ \rho = \text{Ad}(w) \circ \Big(\sum_{j=1}^{K-k-1} \theta_{q_j}^{-1} \circ \rho + \theta_{\tilde{q}}^{-1} \circ \rho\Big).\]  Fix $x \in \varphi(N)'\cap R^\mathcal{U}$ with $||x||\leq 1$.  It will suffice to show that $x = pxp + p^\perp xp^\perp$. It is a direct observation that for every $1 \leq i \leq k$ and $1\leq j \leq K-k-1$, $(vp_iv^*)x(wq_jw^*)$ intertwines $v(\theta_{p_i}^{-1} \circ \pi)v^*$ and $w(\theta_{q_j}^{-1} \circ \rho)w^*$.  Then by Corollary \ref{scaletwine} we have
\[(v(p-\tilde{p})v^*)x(w(p^\perp-\tilde{q})w^*) = \sum_{i=1}^k\sum_{j=1}^{K-k-1} (vp_iv^*)x(wq_jw^*)  = 0.\]  Similarly,
\[(w(p^\perp-\tilde{q})w^*)x(v(p-\tilde{p})v^*) = \sum_{i=1}^k\sum_{j=1}^{K-k-1} (wq_jw^*)x(vp_iv^*) = 0.\]
So 
\begin{align*}
pxp^\perp &= px(w\tilde{q}w^*) + (v\tilde{p}v^*)x(p^\perp - w\tilde{q}w^*)\\
&\text{and}\\
p^\perp x p &= (w\tilde{q}w^*)xp + (p^\perp - w\tilde{q}w^*)x(v\tilde{p}v^*).
\end{align*}
Thus,
\begin{align*}
||pxp^\perp + p^\perp x p||_2 &\leq ||px(w\tilde{q}w^*)||_2 + ||(v\tilde{p}v^*)x(p^\perp - w\tilde{q}w^*)||_2 + ||(w\tilde{q}w^*)xp||_2 \\
&+ ||(p^\perp - w\tilde{q}w^*)x(v\tilde{p}v^*)||_2\\
&\leq ||w\tilde{q}w^*||_2 + ||v\tilde{p}v^*||_2 + ||w \tilde{q}w^*||_2 + ||v\tilde{p}v^*||_2\\
&< 4\sqrt{\frac{1}{K}}.
\end{align*}
Since $K \in \mathbb{N}$ was only selected to be bounded below, this shows that $pxp^\perp + p^\perp x p = 0$.  Thus, $x = pxp + p^\perp x p^\perp$.
\end{proof}

\begin{proof}(of part \eqref{equation} of Theorem \ref{simplexthm})
This statement follows from \eqref{sim} and \eqref{iso}.  Let $\dim(\mathcal{Z}(\pi(N)'\cap R^\mathcal{U})) = n$.  If $n < \infty$, then by \eqref{sim} we have that $F_{[\pi]}$ is an $n$-vertex simplex and thus $\dim(F_{[\pi]}) = n-1$. If $n = \infty$ but $\dim(F_{[\pi]}) < \infty$, then $[\pi]$ is an average of finitely many extreme points.  And this would imply by \eqref{iso} that $\dim(\mathcal{Z}(\pi(N)'\cap R^\mathcal{U}))$ is finite--a contradiction.  So we must have $\dim(F_{[\pi]}) = \infty$.  
\end{proof}


\bibliographystyle{plain}
\bibliography{thesisbib}{}

\end{document}